\documentclass[11pt,a4paper]{article}
\usepackage{amsmath,amssymb,latexsym,graphicx,mathrsfs,amsthm}
\usepackage{secdot}
\usepackage[dvipsnames]{xcolor}

\usepackage{url}
\usepackage{hyperref}
\definecolor{ForestGreen}{rgb}{0.1,0.6,0.05}
\definecolor{EgyptBlue}{rgb}{0.063,0.1,0.6}
\definecolor{RipeOlive}{HTML}{556B2F}
\hypersetup{
	colorlinks=true,
	linkcolor=EgyptBlue,         
	citecolor=ForestGreen,
	urlcolor=olive
}

\usepackage[hyperpageref]{backref}

\usepackage{epstopdf}
\newtheorem{theorem}{Theorem}
\newtheorem{proposition}[theorem]{Proposition}
\newtheorem{lemma}[theorem]{Lemma}
\newtheorem{corollary}[theorem]{Corollary}

\theoremstyle{definition}
\newtheorem{remark}[theorem]{Remark}

\numberwithin{equation}{section}
\numberwithin{theorem}{section}

\usepackage[utf8]{inputenc}
\usepackage{amsfonts}
\usepackage[left=2.5cm,right=2.5cm,top=2cm,bottom=2cm]{geometry}
\usepackage[affil-it]{authblk}

\title{On some unexpected properties of radial and symmetric eigenvalues and eigenfunctions of the $p$-Laplacian on a disk
}

\author{Vladimir Bobkov\thanks{E-mail: \texttt{bobkov@kma.zcu.cz}; Corresponding author}~}
\author{Pavel Dr\'abek\thanks{E-mail: \texttt{pdrabek@kma.zcu.cz}}}
\affil{{\small Department of Mathematics and NTIS, Faculty of Applied Sciences,\\ University of West Bohemia, Univerzitn\'i 8, 306 14 Plze\v{n}, Czech Republic}}
\date{}

\begin{document}
\maketitle 

\begin{abstract}
	We discuss several properties of eigenvalues and eigenfunctions of the $p$-Laplacian on a ball subject to zero Dirichlet boundary conditions. Among main results, in two dimensions, we show the existence of nonradial eigenfunctions which correspond to the radial eigenvalues. Also we prove the existence of eigenfunctions whose shape of the nodal set cannot occur in the linear case $p=2$. Moreover, the limit behavior of some eigenvalues as $p \to 1+$ and $p \to +\infty$ is studied.
	
\end{abstract}

\section{Introduction}\label{intro}
In the present article we consider the nonlinear eigenvalue problem
\begin{equation}
\label{D}
\left\{
\begin{aligned}
-\Delta_p u &= \lambda |u|^{p-2} u &&{\rm in}\ B_1, \\
u&=0 &&{\rm on }\ \partial B_1,
\end{aligned}
\right.
\end{equation}
where $B_1$ is an open unit ball in $\mathbb{R}^N$, $N \geq 2$, and $\Delta_p u := \text{div}(|\nabla u|^{p-2} \nabla u)$ is the  $p$-Laplacian, $p>1$. 

The structure of the set of \textit{all} eigenvalues for the problem \eqref{D} is not known if $p \neq 2$. 
On the other hand, it is well known that \eqref{D} admits several infinite sequences of variational eigenvalues obtained by Ljusternik–Schnirelman theory via  topological indexes such as Krasnosel'skii genus (cf. \cite{azorero}), or the sequences of variational eigenvalues proposed in \cite{drabrob1999} and \cite{perera}. 
We denote the sequence from \cite{drabrob1999} as $\{\lambda_k(B_1; p)\}_{k \in \mathbb{N}}$ (see Section 2 below) for further use.
Let us stress that it is a long-standing open problem if the above mentioned sequences coincide and characterize all eigenvalues of the problem \eqref{D}. 
Furthermore, the behavior of eigenfunctions corresponding to these sequences for general $p>1$ is studied rather rarely and little is known in contrast to the linear case $p = 2$. For instance, it is not known if all these eigenfunctions inherit any symmetries from the ball $B_1$.
Nevertheless, the radiality of $B_1$ allows to \textit{construct} some special  sequences of eigenvalues and corresponding eigenfunctions. One of them consists of radial eigenvalues, i.e.,  numbers $\mu_k(p)$, $k \in \mathbb{N}$, for which the radial version of \eqref{D},
\begin{equation}
\label{Drad}
\left\{
\begin{aligned}
-&(r^{N-1}|u'|^{p-2}u')' = \mu_k(p)  r^{N-1} |u|^{p-2} u, &&r \in (0,1), \\
&u'(0)=0,~~ u(1)=0,
\end{aligned}
\right.
\end{equation}
possesses corresponding nontrivial solution $\Phi_k$. Note that $\mu_k(p) = \nu_k^p(p)$ and, up to an arbitrary multiplier,  $\Phi_k(r) = \Phi(\nu_k(p) r)$, where $\Phi(r)$ is a (unique) solution of the Cauchy problem 
\begin{equation*}
\left\{
\begin{aligned}
-&(r^{N-1}|u'|^{p-2}u')' = r^{N-1} |u|^{p-2} u, &&r > 0, \\
&u'(0)=0,~~ u(0)=1,
\end{aligned}
\right.
\end{equation*}
and $\nu_k(p)$ is the $k$th zero of $\Phi(r)$ (all zeros are simple and $\nu_k(p) \to +\infty$ as $k \to +\infty$), cf. \cite[Lemmas 5.2, 5.3]{delPinoManas}. Consequently, $\Phi_k$ has exactly $k-1$ zeros in $(0,1)$.

Let us note that 
$\lambda_1(B_1; p) = \mu_1(p)$ and $\lambda_2(B_1; p) < \mu_2(p)$ (the last inequality has been proved in \cite{bendrabgirg} for the two-dimensional case and generalized in \cite{andrabsasi} to higher dimensions; for related results concerning the second eigenvalue of the $p$-Laplacian with Neumann boundary conditions see \cite{brasco}). 
In general, for all $p>1$ and $k \geq 2$ it can be easily shown (see Lemma \ref{lem:lleqm} below) that 
$\lambda_k(B_1; p) \leq \mu_k(p)$. However, generalizing the result of \cite{andrabsasi} we prove that actually the strict inequality is valid.
\begin{theorem}\label{thm:l<m}
	$\lambda_k(B_1;p) < \mu_k(p)$ for all $p>1$ and $k \geq 2$.
\end{theorem} 
In other words, none of eigenfunctions corresponding to $\lambda_k(B_1;p)$ can be a radial function with a number of nodal domains greater than or equal to $k$ (or, equivalently, with a number of zeros greater than or equal to $k-1$) whenever $k \geq 2$. However, we emphasize that there can exist a radial eigenfunction corresponding to $\lambda_k(B_1;p)$ with lower number of nodal domains, as it is for the linear case $p=2$ (for example, if $N=2$, then $\lambda_6(B_1;2) = \mu_2(2)$, see below).

Note that the result of Theorem \ref{thm:l<m} carries over to the sequence of Krasnosel'skii eigenvalues \cite{azorero}, since the $k$th Krasnosel'skii eigenvalue for the problem \eqref{D} is less than or equal to $\lambda_k(B_1;p)$ for any $p>1$ and $k \in \mathbb{N}$, see \cite[p.~195]{drabrob1999}.

Another sequence of symmetric eigenvalues $\{\tau_k(p)\}_{k \in \mathbb{N}}$ for the problem \eqref{D} was recently constructed in \cite{andrabsasi} in the following way. Let us fix $k \in \mathbb{N}$ and consider the first eigenpair $(\lambda_1(B_1^{\pi/k};p), \psi_p)$ of the $p$-Laplacian with zero Dirichlet boundary conditions on a spherical wedge of $B_1$ with the dihedral angle $\pi/k$, which we denote by $B_1^{\pi/k}$. Then we construct a function $\Psi_k$ in $B_1$ by gluing together rotated and reflected copies of $\psi_p$, and set  $\tau_k(p) = \lambda_1(B_1^{\pi/k};p)$. 
In \cite[Theorem 1.2]{andrabsasi} it is proved that $(\tau_k(p), \Psi_k)$ is indeed an eigenpair for the problem \eqref{D} for any $k \in \mathbb{N}$. 
By construction, $\Psi_k$ has precisely $2k$ nodal domains and, as a consequence,
$\lambda_{2k}(B_1;p) \leq \tau_k(p)$ for any $p>1$, see Lemma \ref{lem:lleqm} below.

\begin{remark}\label{rem:new_sequence}
The sequence $\{\tau_k(p)\}_{k \in \mathbb{N}}$ can be easily generalized. Let us define a sequence $\{\hat{\tau}_{k,n}(p)\}_{k,n \in \mathbb{N}}$ such that for any $k,n \in \mathbb{N}$ the corresponding eigenfunction $\hat{\Psi}_{k,n}$ is constructed as $\Psi_{k}$ above, but the initial ``building block'' is taken to be any $n$th variational eigenfunction in $B_1^{\pi/k}$ (instead of the first one), and $\hat{\tau}_{k,n}(p) = \lambda_n(B_1^{\pi/k};p)$. (Here for simplicity we use the eigenvalue sequence of type \cite{drabrob1999}, but we can take an ordered union of all known variational sequences as well.)
Then, the proof of \cite[Theorem 1.2]{andrabsasi} can be applied without any changes to show that $\hat{\Psi}_{k,n}$ is indeed an eigenfunction in $B_1$ which corresponds to the eigenvalue $\hat{\tau}_{k,n}(p)$.
Notice that $\hat{\tau}_{k,1}(p) = \tau_k(p)$ and $\hat{\Psi}_{k,1} = \Psi_k$. Moreover, in the sequence $\{\hat{\tau}_{k,n}(p)\}_{k,n \in \mathbb{N}}$ duplicate items can occur.
However, $\{\hat{\tau}_{k,n}(p)\}_{k,n \in \mathbb{N}}$ is in fact richer than $\{\tau_k(p)\}_{k \in \mathbb{N}}$. 
\end{remark}

To formulate our 
\textit{next results, we consider the problem \eqref{D} on the planar disk}, and below in this section we always assume $N=2$.  For visual convenience we will denote the eigenvalues $\tau_1(p)$, $\tau_2(p)$ and $\tau_3(p)$ as $\lambda_{\ominus}(p)$,  $\lambda_{\oplus}(p)$ and $\lambda_{\circledast}(p)$, respectively. 
For the same reason, 
the second radial eigenvalue
$\mu_2(p)$ will be denoted as  $\lambda_{\circledcirc}(p)$.

In the special case $p = 2$ the problem \eqref{D} becomes linear and corresponding  eigenfunctions can be given explicitly in polar coordinates through the Bessel functions:
\begin{equation}\label{phi}
\varphi_{n,k}(r, \theta) = (a \cos(n \theta) + b \sin(n \theta)) J_n(\alpha_{n,k} r),
\end{equation}
where $n \in \mathbb{N}\cup\{0\}$, $k \in \mathbb{N}$, $a, b \in \mathbb{R}$, $J_n$ is the Bessel function of order $n$ and $\alpha_{n,k}$ is the $k$th positive zero of $J_n$. 
The eigenvalue corresponding to $\varphi_{n,k}$ is $\alpha_{n,k}^2$. It follows easily that  $\alpha_{0,k}^2 = \mu_k(2)$ and $\alpha_{n,1}^2 = \tau_{n}(2)$ for any $k, n \in \mathbb{N}$, by construction. In general, orthogonality and completeness of the set of eigenfunctions \eqref{phi} in $L^2(B_1)$ imply that $\{\alpha_{n,k}^2\}_{n,k \in \mathbb{N}} = \{\hat{\tau}_{m,l}(2)\}_{m,l \in \mathbb{N}}$.

On the other hand, the \textit{ordered} set of all eigenvalues, counting multiplicities, of the problem \eqref{D} for $p=2$ can be defined via the classical minimax variational principles. It is not hard to show that this set coincides with $\{\alpha_{n,k}^2\}_{n \in \mathbb{N}\cup \{0\}, k \in \mathbb{N}}$ and  $\{\lambda_k(B_1; 2)\}_{k \in \mathbb{N}}$.
As a consequence, its properties can be studied much easier than for general $p > 1$.
For instance, it is known that
\begin{align*} 
\lambda_{\ominus}(2) &= \lambda_2(B_1; 2) = \lambda_3(B_1; 2), \\
\lambda_{\oplus}(2) &= \lambda_4(B_1; 2) = \lambda_5(B_1; 2), \\
\lambda_{\circledcirc}(2) &= \lambda_6(B_1; 2),\\
\lambda_{\circledast}(2) &= \lambda_7(B_1; 2) = \lambda_8(B_1; 2),
\end{align*}
and
\begin{equation}\label{chain1}
\lambda_\ominus(2) < \lambda_\oplus(2) < \lambda_\circledcirc(2) < \lambda_\circledast(2).
\end{equation}
At the same time, for general
$p>1$ it holds (see Lemma \ref{monot} below)
\begin{equation}\label{chain2}
\lambda_\ominus(p) < \lambda_\oplus(p) < \lambda_\circledast(p).
\end{equation}
The main observation of this article consists in the fact that $\lambda_\circledcirc(p)$, as a function of $p$, is not squeezed between  $\lambda_\oplus(p)$ and  $\lambda_\circledast(p)$ for all $p>1$ as in \eqref{chain1}, see Fig.~\ref{fig:Fig3}.
\begin{theorem}\label{Th1}
	There exist $p_0, p_1 > 1$ such that
	\begin{itemize}
		\item[\normalfont{(i)}] $\lambda_{\ominus}(p) < \lambda_{\circledcirc}(p) < \lambda_{\oplus}(p)$ for all $1<p<p_0$;
		\item[\normalfont{(ii)}] $\lambda_{\ominus}(p) < \lambda_{\oplus}(p) < \lambda_{\circledcirc}(p)$ for all $p > p_1$.
	\end{itemize}
\end{theorem}

Let us remark that since
$\lambda_{\ominus}(p)$, $\lambda_{\oplus}(p)$ and $\lambda_\circledcirc(p)$ are continuous with respect to $p > 1$ (see Lemma \ref{contin} below), the inequalities \eqref{chain1} immediately imply the assertion (ii) in a neighborhood of $p=2$. 
Let us also mention that in the limit case $p=1$ there holds $\lambda_{\ominus}(1) = \lambda_2(B_1; 1)$, see~\cite{Parini0,esposit0}.

Note that for $p=2$ the basis property of eigenfunctions \eqref{phi} along with the fact that $\alpha_{n,k} \neq \alpha_{m,l}$ (cf. \cite[\S 15.28]{watson}) implies that 
each radial eigenvalue $\mu_k(2)$ has only one linearly independent eigenfunction, that is, the multiplicity one, while all other eigenvalues have exactly two linearly independent eigenfunctions, and, consequently, the multiplicity two. 
On the other hand, considering the multiplicity of an eigenvalue for general $p>1$ as a number of linearly independent eigenfunctions, we can only guarantee that the multiplicity of $\mu_k(p)$ is \textit{at least} one, and the multiplicity of any other eigenvalue is \textit{at least} two. 
Indeed, if an eigenfunction $\varphi_p$ of the problem \eqref{D} is not radial, then there exists its rotation $\widetilde{\varphi}_p$ which is linearly independent of $\varphi_p$ and also is an eigenfunction of \eqref{D} by virtue of the rotational invariance of the $p$-Laplacian.

\begin{figure}[!h]
	\centering
	\includegraphics[width=0.8\linewidth]{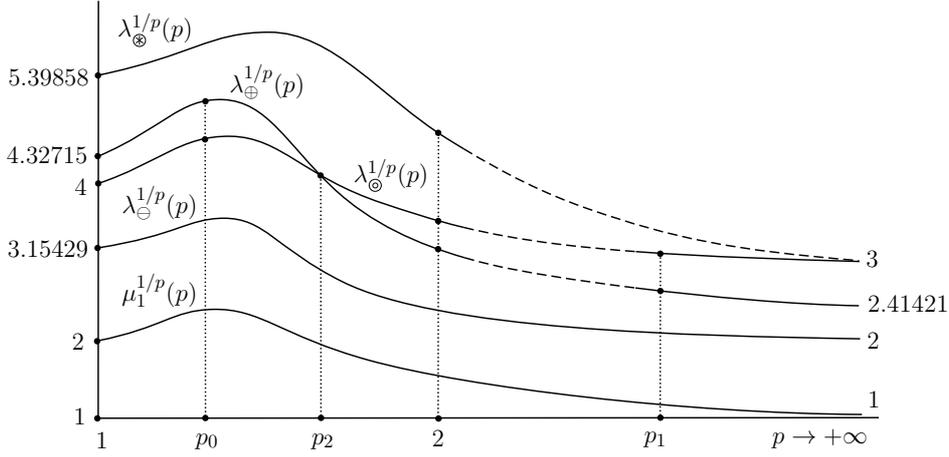}
	\caption{Dependence of several lower eigenvalues raised to the power $1/p$ on $p>1$.}
	\label{fig:Fig3}
\end{figure}

At the same time, due to the continuity of the considered eigenvalues with respect to $p>1$, Theorem \ref{Th1} yields the following  fact.
\begin{corollary}\label{mult3}
	There exists $p_2 \in [p_0, 2)$ such that $\lambda_{\circledcirc}(p_2) = \lambda_{\oplus}(p_2)$. That is, $\lambda_{\oplus}(p_2)$ and $\lambda_{\circledcirc}(p_2)$ have the multiplicities at least three.
\end{corollary}
Indeed, it is not hard to see that a nonradial eigenfunction $\Psi_2$ corresponding to $\lambda_\oplus(p_2)$, its appropriate rotation $\widetilde{\Psi}_2$, and a radial eigenfunction $\Phi_2$ corresponding to $\lambda_\circledcirc(p_2)$ are mutually linearly independent. 
We emphasize that the result like in Corollary \ref{mult3} is impossible for $p=2$, where the maximal multiplicity of eigenvalues on a disk is two.  
Moreover, this fact gives the negative answer to the open problem (3) from \cite{andrabsasi}: \textit{``Is it true that all the eigenfunctions corresponding to the second
radial eigenvalue $\lambda_\circledcirc(p)$ are radial?''}

Theorem \ref{Th1} and Corollary \ref{mult3} can be complemented by the following general fact. Recall that we assume $N=2$.
\begin{theorem}\label{th:eqiv}
	For any integer $k \geq 3$ there exist $n \in \mathbb{N}$ and $p_k > 2$ such that $\mu_k(p_k) = \tau_n(p_k)$.
\end{theorem}

The following asymptotic result will be among the main ingredients in the proofs of the above-mentioned facts.
\begin{proposition}\label{prop:asym1}
	$\lim\limits_{p \to 1+} \mu_k(p) = 2k$ and $\lim\limits_{p \to +\infty} \mu_k^{1/p}(p) = 2k-1$ for any $k \in \mathbb{N}$.
\end{proposition}

Clearly, this proposition implies that $\lim\limits_{p \to 1+} \lambda_\circledcirc(p) = 4$ and $\lim\limits_{p \to +\infty} \lambda_\circledcirc^{1/p}(p) = 3$, which rigorously confirm numerical suggestions from \cite[\S 4.1]{horak} (and \cite{bendrabgirg}). Let us provide an additional information on the asymptotic behavior as $p \to +\infty$. 
\begin{proposition}\label{prop:infty}
	We have $\lim\limits_{p \to +\infty} \mu_k^{1/p}(p)  = \lim\limits_{p \to +\infty}\tau_n^{1/p}(p)$ if and only if $k = 2$ and $n = 3$. In particular, $\lim\limits_{p \to +\infty} \lambda_\circledcirc^{1/p}(p) = \lim\limits_{p \to +\infty} \lambda_\circledast^{1/p}(p) = 3$.
\end{proposition}

From the exact formula \eqref{phi} it can be readily deduced that the nodal set of \textit{any} eigenfunction of the Laplace operator with zero Dirichlet  boundary conditions on a disk has to be of one of the following types:
\begin{enumerate}
	\item \textit{Diameters}. Corresponding eigenvalues are $\alpha_{n,1}^2 (= \tau_n(2))$, $n \in \mathbb{N}$.
	\item \textit{Concentric circles}. Corresponding eigenvalues are $\alpha_{0,k}^2 (=\mu_k(2))$, $k \in \mathbb{N}$. 
	\item \textit{Combination of cases $1$ and $2$}. Corresponding eigenvalues are $\alpha_{n,k}^2$, $n,k \in \mathbb{N}$, $k \geq 2$.
\end{enumerate}
We show that for general $p>1$ there might exist eigenfunctions of \eqref{D} whose nodal set cannot be described by the above-mentioned types of curves.
\begin{theorem}\label{th:not123}
	There exists $p^* > 1$ such that for any $p>p^*$ the nodal set of an eigenfunction $\hat{\Psi}_{4,2}$ defined in Remark~\ref{rem:new_sequence} is not of either type 1, 2 and 3. Namely, the nodal set of $\hat{\Psi}_{4,2}$ is a combination of diameters and a noncircular set.
\end{theorem}

\section{Preliminaries}
Let $\Omega \subset \mathbb{R}^N$, $N \geq 2$, be a bounded domain with the Lipschitz boundary $\partial \Omega$. 
Denote by $\lambda_k(\Omega; p)$ the $k$th eigenvalue of type \cite{drabrob1999} for the $p$-Laplacian on $\Omega$ subject to homogeneous Dirichlet boundary conditions on $\partial \Omega$, that is,
\begin{equation}\label{eigenvalue}
\lambda_k(\Omega; p) := \inf\limits_{\mathcal{A} \in \mathcal{F}_k} \sup\limits_{u \in \mathcal{A}} \int_{\Omega} |\nabla u|^p \, dx.
\end{equation}
Here
\begin{align*}
\mathcal{F}_k &:= \{\mathcal{A} \subset \mathcal{S}:~ \mathcal{A}=h(\mathcal{S}^{k-1}), \text{ where } h: \mathcal{S}^{k-1} \to \mathcal{S} \text{ is a continuous and odd function}\,\}, \\
\mathcal{S} &:= \{u \in W_0^{1,p}(\Omega):~ \|u\|_{L^p(\Omega)} = 1 \}
\end{align*}
and $\mathcal{S}^{k-1}$ is a unit sphere in $\mathbb{R}^{k}$, $k \in \mathbb{N}$.

\begin{lemma}\label{lem:lleqm}
 $\lambda_k(B_1;p) \leq \mu_k(p)$ and $\lambda_{2k}(B_1;p) \leq \tau_k(p)$ for any $p>1$ and $k \in \mathbb{N}$. 
\end{lemma}
\begin{proof}
	We give the proof only for the first inequality, since the arguments rely solely on the number of nodal domain of an eigenfunction corresponding to $\mu_k(p)$ or $\tau_k(p)$. 
	Let $\Phi_k$ be the radial eigenfunction corresponding to $\mu_k(p)$. It is known that $\Phi_k$ has exactly $k$ nodal domains 	(cf. \cite[Proposition 4.1]{delPinoManas}), that is, $\Phi_k = \sum_{i=1}^{k} \Phi_k^i$, where each $\Phi_{k}^i$ is the first eigenfunction of the $p$-Laplacian on the $i$th nodal domain of $\Phi_k$ extended by zero to the rest part of $B_1$. Let us normalize $\Phi_k$ such that $\|\Phi_{k}^i\|_{L^p} = 1$ for all $i \in \{1,2,\dots,k\}$, and consider the set $\mathcal{A}$ and the map $h$ given by
	$$
	\mathcal{A} := \left\{\sum_{i=1}^{k} c_i \Phi_{k}^i:~  \sum_{i=1}^{k} |c_i|^p = 1 \right\}
	\quad \text{and} \quad 
	h(x_1,\dots,x_k) :=  \sum_{i=1}^{k} |x_i|^{\frac{2}{p}-1} x_i \, \Phi_{k}^i.
	$$
	By construction, $\mathcal{A} \subset \mathcal{S}$, $h \in C(\mathcal{S}^{k-1}, \mathcal{A})$,  and $h$ is odd. Therefore, $\mathcal{A} \in \mathcal{F}_{k}$. On the other hand, $(\mu_k(p), \Phi_{k}^i)$ is the first eigenpair of the $p$-Laplacian on each $i$th nodal domain. Therefore, by \eqref{eigenvalue}, 
	$$
	\lambda_k(B_1;p) \leq \sup\limits_{u \in \mathcal{A}} \int_{B_1} |\nabla u|^p \, dx = \mu_k(p).
	$$
\end{proof}

Let us give a characterization of $\mu_k(p)$ which is useful in order to estimate $\mu_k(p)$ from above.
\begin{lemma}\label{varcharrad}
	For any $p>1$ and $k \geq 2$ the eigenvalue $\mu_k(p)$ can be characterized as 
	\begin{equation}\label{newdef}
	\mu_k(p) = \min\limits_{0<r_1<\dots<r_{k-1}<1} \max \left\{
	\mu_1^{(0,r_1)}(p), 
	\mu_1^{(r_1, r_2)}(p), \dots,
	\mu_1^{(r_{k-1},1)}(p)
	\right\},
	\end{equation}
	where $\mu_1^{(0,r_1)}(p)$ is the first (radial) eigenvalue of $-\Delta_p$ on the ball of radius $r_1$, i.e.,
	$$
	\mu_1^{(0,r_1)}(p) = \inf_{\substack{u \in W^{1,p}(0, r_1)\\u(r_1)=0}} \frac{\int_{0}^{r_1} r^{N-1} |u'|^p \, dr}{\int_{0}^{r_1} r^{N-1} |u|^p \, dr},
	$$
	and $\mu_1^{(r_i,r_{i+1})}(p)$ is the first (radial) eigenvalue of $-\Delta_p$ on the spherical shell with inner radius $r_i$ and outer radius $r_{i+1}$, $i \geq 1$, respectively, that is,
	$$
	\mu_1^{(r_i,r_{i+1})}(p) = \inf_{u \in W_0^{1,p}(r_i,r_{i+1})} \frac{\int_{r_i}^{r_{i+1}} r^{N-1} |u'|^p \, dr}{\int_{r_i}^{r_{i+1}} r^{N-1} |u|^p \, dr}.
	$$
\end{lemma}
\begin{proof}
	Fix arbitrary $p>1$ and $k \geq 2$ and let  $\{r_1^n,\dots,r_{k-1}^n\}_{n \in \mathbb{N}}$ be a minimizing sequence of partitions of interval $(0,1)$ for the problem \eqref{newdef} such that $0<r^n_1<\dots<r^n_{k-1}<1$ for $n \in \mathbb{N}$.	
	We set $r_0^n=0$ and $r_k^n = 1$, $n \in \mathbb{N}$.
	
	Since the sequence of $r_1^n$ is bounded, there exists a convergent
	subsequence $r_1^{n_j}$, $j \in \mathbb{N}$. Next, from the bounded sequence $r_2^{n_j}$ we can also extract a convergent subsequence $r_2^{n_{j_l}}$, $l \in \mathbb{N}$, etc. Finally, passing to suitable subsequence $(k-1)$-times, we find a tuple $(r_1,\dots,r_{k-1})$ such that $(r_1^n,\dots,r_{k-1}^n) \to (r_1,\dots,r_{k-1})$ as $n \to +\infty$, and $0\leq r_1 \leq \dots \leq r_{k-1} \leq 1$.
	
	Let us show that for any $i \in \{0,1,\dots,k-1\}$ there exist $c_i>0$ and $N_i > 0$ such that $|r_i^n - r_{i+1}^n| > c_i$ for all $n \geq N_i$. Assume, by contradiction, that $|r_i^n - r_{i+1}^n| \to 0$ as $n \to +\infty$ for some  $i$. 
	Recalling that $0< r_{i}^n < r_{i+1}^n < 1$, we obviously get $\mu_1^{(r_i^n, r_{i+1}^n)}(p) \to +\infty$. This implies that the sequence $\{r_1^n,\dots,r_{k-1}^n\}_{n \in \mathbb{N}}$ is not a minimizing sequence, a contradiction. As a consequence, we get $0<r_1<\dots<r_{k-1}<1$.
	
	Let us show now that $\mu_1^{(0,r_1)}(p) = 
	\mu_1^{(r_1, r_2)}(p) = \dots =
	\mu_1^{(r_{k-1},1)}(p)$. Suppose, contrary to our claim, that there exists $i \in \{0,1,\dots,k-1\}$ such that either $\mu_1^{(r_i, r_{i+1})}(p) > \mu_1^{(r_{i+1}, r_{i+2})}(p)$ or $\mu_1^{(r_i, r_{i+1})}(p) > \mu_1^{(r_{i-1}, r_{i})}(p)$. 
	Assume, without loss of generality, that the second alternative occurs. 
	It is not hard to prove that $\mu_1^{(r_{i-1}, r)}(p)$ and $\mu_1^{(r, r_{i+1})}(p)$ are continuous with respect to $r \in (r_{i-1}, r_{i+1})$. 
	(Indeed, it follows, e.g., from the general result \cite[Theorem 1]{garciamelian} by choosing appropriate diffeomorphisms.)
	Hence, the strict domain monotonicity of the first eigenvalue of the $p$-Laplacian (cf. \cite[Theorem 2.3]{Alleg}) implies the existence of $r_i^0 < r_i$ such that
	$$
	\mu_1^{(r_i, r_{i+1})}(p) > 
	\mu_1^{(r_i^0, r_{i+1})}(p) = 
	\mu_1^{(r_{i-1}, r_{i}^0)}(p) >
	\mu_1^{(r_{i-1}, r_{i})}(p).
	$$
	If the maximum in \eqref{newdef} is achieved only for one $i$, then we get a contradiction, since we  found better partition with $r_i$ replaced by $r_i^0$. Otherwise, we apply such procedure for other index $i$ and eventually obtain a tuple $(r_1^0, \dots, r_{k-1}^0)$ which optimizes \eqref{newdef} better than $(r_1, \dots, r_{k-1})$.
	
	Let $\varphi_p^{(r_i, r_{i+1})}\in C^1[r_i, r_{i+1}]$ be a positive eigenfunction corresponding to $\mu_1^{(r_i, r_{i+1})}(p)$, $i \in \{0,1,\dots,k-1\}$.
	Since $\varphi_p^{(0, r_1)}$ is unique up to an arbitrary multiplier, we normalize it, without loss of generality, such that $\varphi_p^{(0, r_1)}(0) = 1$.
	Consider the function
	$$
	\varphi_p(r) = \varphi_p^{(0, r_1)}(r) + a_1 \varphi_p^{(r_1, r_2)}(r) + \dots + a_{k-1} \varphi_p^{(r_{k-1}, 1)}(r),
	$$
	where $a_i \in \mathbb{R}$ are chosen to get $\varphi_p \in C^1[0,1]$.
	Notice that such $a_i$ exist. Indeed, first we choose $a_1$ to satisfy $\left(\varphi_p^{(0, r_1)}\right)'(r_1) = a_1 \left(\varphi_p^{(r_1, r_2)}\right)'(r_1)$. Analogously, we choose $a_2$, etc. Let us remark that $\varphi_p(r)$ has $k-1$ zeros, by construction.
	
	Since each 	$\varphi_p^{(r_i, r_{i+1})}$ is a solution of $-(r^{N-1}|u'|^{p-2}u')' = \mu_1^{(r_i, r_{i+1})}(p)  r^{N-1} |u|^{p-2}u$ on $(r_i, r_{i+1})$, we derive that $\varphi_p$ is a solution of \eqref{Drad} with the eigenvalue $\mu = \mu_1^{(r_i, r_{i+1})}(p)$, $i \in \{0,1,\dots,k-1\}$. On the other hand, a unique (up to a multiplier) nonzero solution of \eqref{Drad} with $k-1$ zeros must correspond to $\mu_k(p)$, cf. \cite[Proposition 4.1]{delPinoManas}. Therefore, we obtain the desired result.
\end{proof}

\begin{lemma}\label{contin}
	For any $k \in \mathbb{N}$, 
	$\mu_k(p)$ and $\tau_{k}(p)$ are continuous with respect to $p>1$.
\end{lemma}
\begin{proof}
	The continuity of $\mu_k(p)$ was proved in \cite[Proposition 4.1]{delPinoManas}. 
	The continuity of $\tau_{k}(p)$ follows from
	two facts. First, $\tau_k(p) = \lambda_1(B_1^{\pi/k};p)$, where $B_1^{\pi/k}$ is a spherical wedge of $B_1$ with the dihedral angle $\pi/k$ (see Section \ref{intro}). Second, $\lambda_1(B_1^{\pi/k};p)$ is continuous with respect to $p>1$, cf. \cite{Parini} (and \cite{huang} for smooth domains).
\end{proof}

The next fact  follows directly from the equality $\tau_k(p) = \lambda_1(B_1^{\pi/k};p)$ and the strict domain monotonicity of $\lambda_1(\Omega;p)$ (see, e.g., \cite[Theorem 2.3]{Alleg}).
\begin{lemma}\label{monot}
	$\tau_k(p) < \tau_{k+1}(p)$ for any $p > 1$ and $k \in \mathbb{N}$. In particular,	
	$\lambda_\ominus(p) < \lambda_\oplus(p) < \lambda_\circledast(p)$.
\end{lemma}

Let us recall several facts concerning the limit behavior of eigenvalues as $p \to 1+$ and $p \to +\infty$.
In \cite[Corollary 6]{kawfrid} it has been shown that
\begin{equation}\label{eq:cheeger}
\lim\limits_{p \to 1+} \lambda_1(\Omega;p) = h(\Omega),
\end{equation}
where $h(\Omega)$ denotes the so-called \textit{Cheeger constant} of $\Omega \subset \mathbb{R}^N$ which can be defined as
\begin{equation}\label{cheeger1}
h(\Omega) :=  \inf \frac{|\partial A|}{|A|}.
\end{equation}
Here the infimum is taken over all smooth domains $A$ compactly contained in $\Omega$; $|\partial A|$ is the $(N-1)$-dimensional Hausdorff measure of $\partial A$ and $|A|$ is the Lebesgue measure of $A$.
Any minimizer of \eqref{cheeger1} is called \textit{Cheeger set} of $\Omega$.

\begin{lemma}\label{lem:tauto1}
	Let $N = 2$ and $k \in \mathbb{R}$, $k \geq 1$. Then $\lim\limits_{p \to 1+} \lambda_1(B_1^{\pi/k};p) = \frac{1}{r_k}$, where $r_k$ is the unique root of the equation \eqref{eq:cheegereq} (see below) in the interval $\left(0, \frac{\sin\left(\frac{\pi}{2k}\right)}{1+\sin\left(\frac{\pi}{2k}\right)}\right)$. 
	In particular, 
	$\lim\limits_{p \to 1+} \tau_k(p) = \frac{1}{r_k}$ for all $k \in \mathbb{N}$, and 
	$\lim\limits_{p \to 1+} \lambda_\ominus(p) \approx 3.15429$, $\lim\limits_{p \to 1+} \lambda_\oplus(p) \approx 4.32715$,  $\lim\limits_{p \to 1+} \lambda_\circledast(p) \approx 5.39858$.
\end{lemma}
\begin{proof}
	Let us fix an arbitrary $k \in \mathbb{R}$ such that $k \geq 1$. According to \eqref{eq:cheeger}, we need to find $h(B_1^{\pi/k})$. 	
	Notice that the theory of Cheeger constants and sets in the plain is reasonably well-developed, cf. \cite{kawfrid,kawlanch,krejcirik} and references therein.
	Although the statement of the lemma can be proved in several ways, we use the fact 
	that sector $B_1^{\pi/k}$ is convex for any $k \in \mathbb{R}$, $k \geq 1$, which allows us to apply the characterization of the corresponding Cheeger constant from \cite{kawlanch}. Namely, for any $r \geq 0$, let us define the following set (see Fig.~\ref{fig:Fig1}):
	$$
	B_{1,r}^{\pi/k} := \{ x \in B_{1}^{\pi/k}:~ \text{dist}(x, \partial B_{1}^{\pi/k}) > r \}.
	$$
	We see that $B_{1,r}^{\pi/k}$ is nonempty, if $r$ is less than the radius of a maximal disk inscribed in $B_{1}^{\pi/k}$, that is, $r \in \left(-\infty, \frac{\sin\left(\frac{\pi}{2k}\right)}{1+\sin\left(\frac{\pi}{2k}\right)}\right)$.
\begin{figure}[ht]
	\begin{minipage}[t]{0.49\linewidth}
		\centering
		\includegraphics[width=0.9\linewidth]{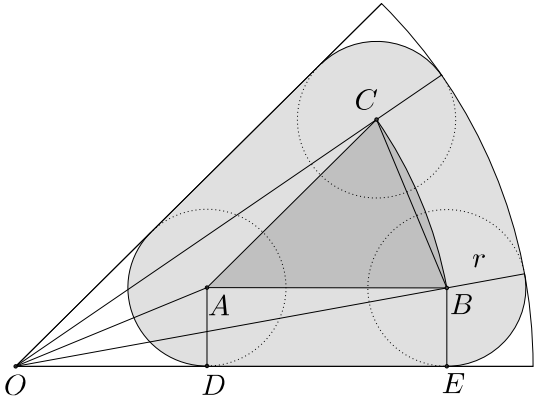}
		\caption{The Cheeger set of $B_1^{\pi/4}$ (light gray) and $B_{1,r}^{\pi/4}$ (dark gray).}
		\label{fig:Fig1}
	\end{minipage}
	\hfill
	\begin{minipage}[t]{0.49\linewidth}
		\centering
		\includegraphics[width=0.9\linewidth]{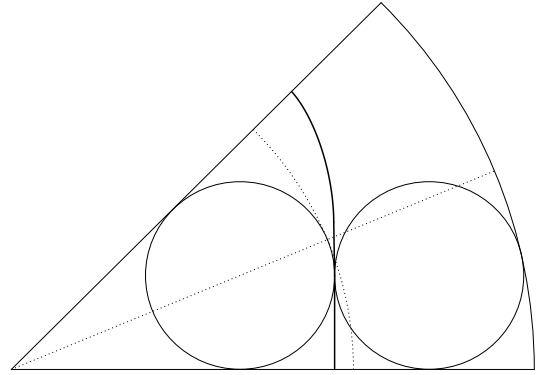}
		\caption{$p \gg 2$. Best packing of two equal circles in $B_1^{\pi/4}$ and conjectured nodal set.}
		\label{fig:Fig2}
	\end{minipage}
\end{figure}
	Then, \cite[Theorem 1]{kawlanch} implies the existence of unique $r_k > 0$ such that $|B_{1,r_k}^{\pi/k}| = \pi r_k^2$ and the Cheeger constant satisfies $h(B_{1}^{\pi/k}) = \frac{1}{r_k}$. The area of $B_{1,r}^{\pi/k}$ can be easily found by direct geometrical calculations. Indeed, $|B_{1,r}^{\pi/k}| = |\triangle ABC| + |BC_{cs}|$, where by $|BC_{cs}|$ we denote the area of the circular segment bounded by arc $\overset\frown{BC}$ of radius $|OB| = 1-r$ and chord $BC$. Moreover, $|AB| = |DE| = \sqrt{1-2r} - r \cot\left(\frac{\pi}{2k}\right)$, 
	$\angle BOC = \frac{\pi}{k} - 2 \arcsin\left(\frac{r}{1-r}\right)$, $\angle BAC = \frac{\pi}{k}$. Therefore, using these values, we arrive at the desired equation for $r$:
	\begin{equation}\label{eq:cheegereq}
	\frac{1}{2}|AB|^2 \sin(\angle BAC) + \frac{(1-r)^2}{2}(\angle BOC + \sin(\angle BOC)) = \pi r^2.
	\end{equation}
\end{proof}

The limit behavior as $p \to +\infty$ has been studied in \cite{julindman1999}, and in \cite[Lemma 1.5]{julindman1999} it was proved that
\begin{equation}\label{eq:p-to-infty}
\lim\limits_{p \to +\infty} \lambda_1^{1/p}(\Omega;p) = \frac{1}{\max\limits_{x \in \Omega} \text{dist}(x, \partial \Omega)} \equiv \frac{1}{r_\Omega},
\end{equation}
where $r_\Omega$ is the radius of a maximal ball inscribed in  $\Omega \subset \mathbb{R}^N$.

\begin{lemma}\label{lem:tautoinfty}
Let $N = 2$ and $k \in \mathbb{R}$, $k \geq 1$. Then 
$\lim\limits_{p \to +\infty} \lambda_1^{1/p}(B_1^{\pi/k};p) = \frac{1+\sin\left(\frac{\pi}{2k}\right)}{\sin\left(\frac{\pi}{2k}\right)}$. In particular, 
$\lim\limits_{p \to +\infty} \tau_k^{1/p}(p) = \frac{1+\sin\left(\frac{\pi}{2k}\right)}{\sin\left(\frac{\pi}{2k}\right)}$ for all $k \in \mathbb{N}$, and $\lim\limits_{p \to +\infty} \lambda^{1/p}_\ominus(p) = 2$, $\lim\limits_{p \to +\infty} \lambda^{1/p}_\oplus(p) \approx 2.41421$,  $\lim\limits_{p \to +\infty} \lambda^{1/p}_\circledast(p) = 3$.
\end{lemma}
\begin{proof}
	These facts easily follow from \eqref{eq:p-to-infty} and the corresponding formula for the radius of a maximal disk inscribed in sector $B_1^{\pi/k}$.
\end{proof}

\section{Proof of the main results}

\subsection{Proof of Theorem \ref{thm:l<m}}
The proof crucially relies on the proof of \cite[Theorem 1.1]{andrabsasi} and appears to be its direct generalization to the case of higher radial eigenvalues. 
As in the proof of Lemma \ref{lem:lleqm}, let $\Phi_k = \sum_{i=1}^{k} \Phi_k^i$ be the radial eigenfunction corresponding to $\mu_k(p)$ and each $\Phi_k^i$ be of a constant sign. We may assume that $(\mu_k(p), \Phi_k^1)$ is the first eigenpair of the $p$-Laplacian on a ball $B_{r_1}$ and $(\mu_k(p), \Phi_k^i)$ is the first eigenpair of the $p$-Laplacian on a spherical shell $B_{r_{i}}\setminus \overline{B_{r_{i-1}}}$, $i = \{2,3,\dots,k\}$, and each $\Phi_k^i$ is extended by zero outside of its support. Moreover, we normalize all $\Phi_k^i$ to satisfy $\|\Phi_k^i\|_{L^p}=1$. All balls $B_{r_i}$ are concentric and assumed to have the center at the origin.

Let us now shift the ball $B_{r_1}$ along the coordinate vector $e_1 = (1,0,\dots,0)$ on value $t_n$, where $t_n  \to r_2-r_1$ as $n \to +\infty$. Due to translation invariance of the $p$-Laplacian, the first eigenvalue of the shifted ball $B_{r_1}(t_n e_1)$ is equal to $\mu_k(p)$ and the corresponding eigenfunction $u_n$ coincides with $\Phi_k^1$ up to a translation. 
Next, denote by $(\kappa_n, v_n)$ the first eigenpair of the $p$-Laplacian on the eccentric spherical shell $B_{r_2} \setminus \overline{B_{r_1}(t_n e_1)}$, $n \in \mathbb{N}$. 
By the result of \cite{ChorMah} it holds $\kappa_n \leq \mu_k(p)$. 
To make the notations closer to \cite[Section 3]{andrabsasi} we denote by $\widetilde{u}_n$ and $\widetilde{v}_n$ the zero extensions of $u_n$ and $v_n$ to the entire $B_1$, respectively. 
Finally, the rest of the components $\Phi_k^{i}$, $i = \{3,\dots, k\}$, remains fixed. 

Now, for every $n \in \mathbb{N} \cup \{0\}$, we consider the set 
$$
\mathcal{A}_n := \{ a \widetilde{u}_n + b \widetilde{v}_n + \sum_{i=3}^{k} c_i \Phi_k^i:~ |a|^p + |b|^p + \sum_{i=3}^{k} |c_i|^p = 1 \}.
$$
By construction, $\mathcal{A}_n \subset \mathcal{S}$. Moreover, arguing as in the proof of Lemma \ref{lem:lleqm} it can be shown that $\mathcal{A}_n \in \mathcal{F}_k$ and $\sup_{u \in \mathcal{A}_n} \int_{B_1} |\nabla u|^p \, dx = \mu_k(p)$ for all $n \in \mathbb{N} \cup \{0\}$. 
Then, the proof of \cite[Lemma 3.2]{andrabsasi} can be applied without major changes to prove the existence of $n_0 \in \mathbb{N}$ such that $\mathcal{A}_{n_0}$ does not contain any critical point of $\int_{B_1} |\nabla u|^p \, dx$ on $\mathcal{S}$. Therefore, \cite[Proposition 2.2]{andrabsasi} implies the existence of $\widetilde{\mathcal{A}}_{n_0} \in \mathcal{F}_k$ such that 
$$
\sup_{u \in \widetilde{\mathcal{A}}_{n_0}} \int_{B_1} |\nabla u|^p \, dx < \sup_{u \in \mathcal{A}_{n_0}} \int_{B_1} |\nabla u|^p \, dx = \mu_k(p), 
$$
which leads to the desired conclusion.

\subsection{Proof  of Proposition \ref{prop:asym1}}
First we show that $\lim\limits_{p \to 1+} \mu_k(p) = 2 k$. 
Since $\mu_1(p) = \lambda_1(B_1;p)$, the result for $k=1$ follows from formula \eqref{eq:cheeger} and the fact that $h(B_1) = 2$, cf. \cite{kawfrid}.
Fix  arbitrary $k \geq 2$ and consider the following function from $W_0^p(0,1)$:
$$
v_{\varepsilon}(r) = 
\left\{
\begin{aligned}
&\frac{r}{\varepsilon}, 
&&r \in (0, \varepsilon],\\
&1, 
&&r \in (\varepsilon, 1-\varepsilon],\\
&\frac{1-r}{\varepsilon}, 
&&r \in (1-\varepsilon, 1),
\end{aligned}
\right.
$$
where $\varepsilon \in (0, 1/2)$.
Let $u_{\varepsilon, i}(r) := v_\varepsilon(kr-i)$ be series of translates and dilates of $v_{\varepsilon}(r)$, where $i \in \{0,1,...,k-1\}$. 
Then $u_{\varepsilon, i} \in W_0^{1,p}(\frac{i}{k}, \frac{i+1}{k})$ 
and characterization \eqref{newdef} implies that
$$
\mu_k(p) \leq \max\left\{ 
\frac{\int_{0}^{\frac{1}{k}} r |u_{\varepsilon,0}'|^p \, dr}{\int_{0}^{\frac{1}{k}} r |u_{\varepsilon,0}|^p \, dr}, \dots,
\frac{\int_{\frac{k-1}{k}}^{1} r |u_{\varepsilon,k-1}'|^p \, dr}{\int_{\frac{k-1}{k}}^{1} r |u_{\varepsilon,k-1}|^p \, dr}
\right\}.
$$
The straightforward calculations yield
\begin{align*}
\int_{\frac{i}{k}}^{\frac{i+1}{k}} r |u_{\varepsilon,i}'|^p \, dr = 
k^{p-2} \int_0^1 (s+i) |v_\varepsilon'|^p \, ds = \frac{k^{p-2}}{\varepsilon^{p-1}}
(2i+1)
\end{align*}
and 
\begin{align*}
\int_{\frac{i}{k}}^{\frac{i+1}{k}} r |u_{\varepsilon,i}|^p \, dr = 
k^{-2} \int_0^1 (s+i) |v_\varepsilon|^p \, ds
> k^{-2} \int_\varepsilon^{1-\varepsilon} (s+i)\, ds = 
\frac{k^{-2}}{2}(2i+1)(1-2 \varepsilon).
\end{align*}
Therefore, choosing, for example, $\varepsilon = p-1 < 1/2$, we obtain
$$
\limsup\limits_{p\to 1+}
\frac{\int_{\frac{i}{k}}^{\frac{i+1}{k}} r |u_{p-1,i}'|^p \, dr}{\int_{\frac{i}{k}}^{\frac{i+1}{k}} r |u_{p-1,i}|^p \, dr} 
\leq 2 k 
$$
for any $i\in \{0,1,\dots,k-1\}$, and hence 
$\limsup\limits_{p\to 1+}\mu_k(p) \leq 2 k$.

Suppose, by contradiction, that $\liminf\limits_{p \to 1+} \mu_k(p) < 2 k$. 
Let us take a sequence $\{p_n\}_{n \in \mathbb{N}}$ such that $p_n \to 1+$ as $n \to +\infty$ and $\lim\limits_{n \to +\infty} \mu_k(p_n) < 2 k$. 
Denote by $\{r_1^n, \dots, r_{k-1}^n\}$ the set of zeros of $\Phi_k$ corresponding to $\mu_k(p_n)$ and put $r_0^n = 0$, $r_k^n = 1$. Arguing as at the beginning of the proof of Lemma \ref{varcharrad} it can be shown that $(r_1^n, \dots, r_{k-1}^n)$
converges to some tuple $(r_1, \dots, r_{k-1})$ as $n \to +\infty$, up to an appropriate subsequence, and $0 \leq r_1 \leq \dots \leq r_{k-1} \leq 1$. 
Using the lower estimate for the first eigenvalue $\lambda_1(\Omega;p)$ from \cite{lefton} and the explicit value of the Cheeger constant for annuli (see, for instance, \cite[Section 4]{krejcirik}), we get
$$
\mu_k(p_n) = \lambda_1(B_{r_{i+1}^n} \setminus \overline{B_{r_{i}^n}};p_n) \geq \left(\frac{h(B_{r_{i+1}^n} \setminus \overline{B_{r_{i}^n}})}{p_n}\right)^{p_n} = 
\left(\frac{2}{p_n \left(r_{i+1}^n - r_{i}^n\right)}\right)^{p_n}
$$
for any $n \in \mathbb{N}$ and $i \in \{0,1,\dots,k-1\}$. 
Our assumption and a passage to the limit as $n \to +\infty$ imply that
$$
2k > \lim\limits_{n \to +\infty} \mu_k(p_n) \geq \frac{2}{r_{i+1} - r_{i}}.
$$
Therefore, $r_{i+1} - r_{i} > \frac{1}{k}$ for any $i \in \{0,1,\dots,k-1\}$ and, consequently,
$1 = \sum_{i=0}^{k-1} (r_{i+1} - r_{i}) > 1$,
which is absurd. Thus, $\lim\limits_{p \to 1+} \mu_k(p) = 2 k$. 

Let us show now that $\lim\limits_{p \to +\infty} \mu_k^{1/p}(p) = 2 k - 1$. 
We will implement a similar technique as above. Consider the function
$$
w(r) = 
\left\{
\begin{aligned}
&1+r, 
&&r \in (-1, 0],\\
&1-r, 
&&r \in (0, 1),
\end{aligned}
\right.
$$
and let $u_i(r) := w((2k-1)r-2i)$ for $i \in \{0,1,\dots,k-1\}$. Then, 
$$
u_0 \in W^p\left(0,\frac{1}{2k-1}\right)
\quad \text{and} \quad 
u_i \in W_0^p\left(\frac{2i-1}{2k-1},\frac{2i+1}{2k-1}\right)
$$
for $i \geq 1$. By direct calculations, we obtain
$$
\int_{0}^{\frac{1}{2k-1}} r |u_0'|^p \, dr  = \frac{(2k-1)^{p-2}}{2}
\quad \text{and} \quad 
\int_{\frac{2i-1}{2k-1}}^{\frac{2i+1}{2k-1}} r |u_i'|^p \, dr  = 4i(2k-1)^{p-2}.
$$
On the other hand,
$$
\int_{0}^{\frac{1}{2k-1}} r |u_0|^p \, dr  = \frac{(2k-1)^{-2}}{(p+1)(p+2)}
\quad \text{and} \quad 
\int_{\frac{2i-1}{2k-1}}^{\frac{2i+1}{2k-1}} r |u_i|^p \, dr  = \frac{4i(2k-1)^{-2}}{p+1}.
$$
Therefore, Lemma \ref{varcharrad} implies that
$$
\mu_k(p) \leq (2k-1)^{p} (p+1) \max\left\{\frac{p+2}{2}, 1\right\}.
$$
Raising to the power $1/p$ and passing to the limit as $p \to +\infty$, we derive that $\limsup\limits_{p \to +\infty} \mu_k^{1/p}(p) \leq 2k-1$.

Suppose, contrary to our claim, that $\liminf\limits_{p \to +\infty} \mu_k^{1/p}(p) < 2k-1$. As above, we
take a sequence $\{p_n\}_{n \in \mathbb{N}}$ such that $p_n \to +\infty$ as $n \to +\infty$ and $\lim\limits_{n \to +\infty} \mu_k^{1/{p_n}}(p_n) < 2k-1$, denote by $\{r_1^n, \dots, r_{k-1}^n\}$ the set of zeros of $\Phi_k$ and put $r_0^n = 0$, $r_k^n = 1$. Furthermore, $(r_1^n, \dots, r_{k-1}^n)$
converges to $(r_1, \dots, r_{k-1})$ as $n \to +\infty$, up to an appropriate subsequence, and $0 \leq r_1 \leq \dots \leq r_{k-1} \leq 1$.
Formally speaking, now we want to apply  \eqref{eq:p-to-infty} to conclude that 
\begin{equation}\label{eq:lim:1}
2k-1 > \lim\limits_{n \to +\infty} \mu_k^{1/{p_n}}(p_n) = \lim\limits_{n \to +\infty} \lambda_1^{1/{p_n}}(B_{r_{1}^{n}};p_n) =
\frac{1}{r_1}
\end{equation}
for $i = 0$, and 
\begin{equation}\label{eq:lim:2}
2k-1 > \lim\limits_{n \to +\infty} \mu_k^{1/{p_n}}(p_n) = \lim\limits_{n \to +\infty} \lambda_1^{1/{p_n}}(B_{r_{i+1}^n} \setminus \overline{B_{r_{i}^n}};p_n) =
\frac{2}{r_{i+1} - r_i}
\end{equation}
for $i \in \{1,2,\dots,k-1\}$. However, \eqref{eq:p-to-infty} is valid only for fixed domains. Therefore, to prove \eqref{eq:lim:1} and $\eqref{eq:lim:2}$ we need some sort of continuity with respect to a change of a domain.
Note that simple scaling gives 
$$
\mu_k(p_n) =
\lambda_1(B_{r_{i+1}^n} \setminus \overline{B_{r_{i}^n}};p_n)
= 
\frac{1}{\left(r_{i+1}^n\right)^{p_n}} \, \lambda_1(B_1 \setminus \overline{B_{r_{i}^n/r_{i+1}^n}};p_n).
$$ 
Therefore, if $i=0$, then \eqref{eq:p-to-infty} yields $\lim\limits_{n \to +\infty} \mu_k^{1/{p_n}}(p_n) = \frac{1}{r_1}$, which implies that $r_1 > \frac{1}{2k-1}$, by the assumption. Let $i \in \{1,2,\dots,k-1\}$.
Denote for simplicity $\rho_n := r_{i}^n/r_{i+1}^n$.
Since $0 < r_{i}^n < r_{i+1}^n$, $\rho_n$ is bounded and can be assumed monotonically  convergent to $\rho := r_i/r_{i+1}$ as $n \to +\infty$, up to an appropriate subsequence. 
Suppose first that $\rho_n$ is decreasing, that is, 
$$
B_1 \setminus \overline{B_{\rho_m}}
\subset 
B_1 \setminus \overline{B_{\rho_n}}
\subset 
B_1 \setminus \overline{B_{\rho}}
$$
for all $n,m \in \mathbb{N}$ such that $m < n$.  Then the domain monotonicity of the first eigenvalue of the $p$-Laplacian implies that
\begin{align*}
\lambda_1(B_1 \setminus \overline{B_{\rho_m}};p_n)
\geq
\lambda_1(B_1 \setminus \overline{B_{\rho_n}};p_n)
\geq
\lambda_1(B_1 \setminus \overline{B_{\rho}};p_n).
\end{align*}
Raising to the power $1/p_n$, multiplying by $1/r_{i+1}^n$, passing to the limit as $n \to +\infty$ and using formula \eqref{eq:p-to-infty}, we arrive at
\begin{equation*}
\frac{2}{r_{i+1}(1 - \rho_m)} \geq 
\lim\limits_{n \to +\infty}
\frac{1}{r_{i+1}^n} \, \lambda_1^{1/{p_n}}(B_1 \setminus \overline{B_{\rho_n}};p_n)
\geq 
\frac{2}{r_{i+1}(1 - \rho)}
\end{equation*}
for all $m \in \mathbb{N}$. 
Taking now the limit as $m \to +\infty$, we conclude that \eqref{eq:lim:2} holds, i.e., $r_{i+1}-r_i > \frac{2}{2k-1}$ for $i \in \{1,2,\dots, k-1\}$. 
If the sequence $\rho_n$ is increasing, we proceed analogously.
Finally, using the obtained estimates, we derive that $1 = \sum_{i=0}^{k-1} (r_{i+1} - r_i) > 1$, a contradiction. Therefore, $\lim\limits_{p \to +\infty} \mu_k^{1/p}(p) = 2 k - 1$.

\subsection{Proof Theorem \ref{Th1}}

First we prove statement (i). Since \eqref{chain2} is satisfied for any $p>1$ and inequality $\lambda_{\ominus}(p) < \lambda_{\circledcirc}(p)$ is, in fact,
proved in \cite{bendrabgirg}, 
it remains to prove the existence of $p_0 \in (1,2)$ such that $\lambda_{\circledcirc}(p) < \lambda_{\oplus}(p)$ for all $p \in (1, p_0)$.
In view of the continuity of $\lambda_\circledcirc(p)$ and $\lambda_\oplus(p)$ with respect to $p>1$ (see Lemma \ref{contin}), it becomes sufficient to get
$$
\lim\limits_{p\to 1+} \lambda_\circledcirc(p) < 
\lim\limits_{p\to 1+} \lambda_\oplus(p).
$$
However, this inequality directly follows from Proposition \ref{prop:asym1} and Lemma \ref{lem:tauto1}.

Similarly, to prove statement (ii) it is sufficient to obtain
$$
\lim\limits_{p\to +\infty} \lambda_\oplus^{1/p}(p) < 
\lim\limits_{p\to +\infty} \lambda_\circledcirc^{1/p}(p).
$$
This inequality follows from Proposition \ref{prop:asym1} and Lemma \ref{lem:tautoinfty}.

\subsection{Proof of Theorem \ref{th:eqiv}}

Let us put $n(k) = \lfloor \pi (k-1)\rfloor - 1 \in \mathbb{N}$ for each integer $k \geq 3$ (see the table below for its first several values), where $\lfloor \pi (k-1) \rfloor$ denotes the integer part of $\pi (k-1)$. The idea of the proof is to show that $\mu_k(2) < \tau_{n(k)}(2)$ and $\lim\limits_{p\to +\infty} \mu_k^{1/p}(p) > \lim\limits_{p\to +\infty} \tau_{n(k)}^{1/p}(p)$ for all $k \geq 3$. Then, the continuity of $\mu_k(p)$ and $\tau_{n(k)}(p)$ with respect to $p>1$ will imply the existence of $p_k$ such that $\mu_k(p_k) = \tau_{n(k)}(p_k)$.
\begin{center}
	\begin{tabular}{| c || c | c | c | c | c | c | c | c | c | c | c | c | c | c | c |}
		\hline
		$k$ & 3 & 4 & 5 & 6 & 7 & 8 & 9 & 10 & 11 & 12 & 13 & 14 & 15 & 16 \\
		\hline
		$n(k)$ & 5 & 8 & 11 & 14 & 17 & 20 & 24 & 27 & 30 & 33 & 36 & 39 & 42 & 46\\
		\hline
	\end{tabular}
\end{center}

First we prove that $\mu_k(2) < \tau_{n(k)}(2)$. 
Recall that $\alpha_{m,l}$ denotes the $l$th positive zero of the Bessel function of order $m$ (see Section \ref{intro}), where, in general, $m \in \mathbb{R}$ and $l \in \mathbb{N}$. 
By construction, $\mu_k(2) = \alpha_{0,k}^2$ and $\tau_{n(k)}(2) = \alpha_{n(k),1}^2$, and hence it is sufficient to obtain $\alpha_{0,k} < \alpha_{n(k),1}$.
It can be easily checked that this inequality is valid for $k=3,4,5$. Indeed, the corresponding values of $n(k)$ are $5, 8, 11$, respectively, and
\begin{align*}
8.65372 \approx \alpha_{0,3} &< \alpha_{5,1} ~\approx 8.77148,\\
11.79153 \approx \alpha_{0,4} &< \alpha_{8,1} ~\approx 12.22509,\\
14.93091 \approx \alpha_{0,5} &< \alpha_{11,1} \approx 15.58984.
\end{align*}
Let us show now that 
\begin{equation}\label{eq:a<a}
\alpha_{0,k} < \alpha_{n(k),1} 
\quad
\text{for all} \quad 
k \geq 6.
\end{equation}
It was proved in \cite{infantis} that zeros $\alpha_{m_1,l}$ and $\alpha_{m_2,l}$ are related by  inequality $\alpha_{m_1,l} - m_1 > \alpha_{m_2,l} - m_2$ for any $m_1, m_2 \in \mathbb{R}$ such that  $m_1>m_2>-1$, and $l \in \mathbb{N}$.
Thus, taking $m_1 = 1/2$, $m_2 = 0$ and $l=k$, we get the upper bound $\alpha_{0,k} < \pi k - 1/2$ for all $k \in \mathbb{N}$, see \cite[(4.9)]{infantis}. On the other hand, taking $m_1 = n(k)$, $m_2 = 14$ and $l = 1$, we get the lower bound $\alpha_{n(k),1} > 4.8 + n(k)$ for any $n(k) \geq 14$, since $\alpha_{14,1} \approx 18.89999 > 18.8$. 
Hence, in order to establish \eqref{eq:a<a}, it is enough to verify that $\pi k - 1/2 < 4.8 + n(k)$ for all $k \geq 6$ (or, equivalently, $n(k) \geq 14$). But this inequality follows from 
$$
n(k) = \lfloor \pi (k-1)\rfloor - 1 \geq
\pi(k-1) - 2 > \pi k -1/2 - 4.8.
$$
Summarizing, we have $\mu_k(2) < \tau_{n(k)}(2)$ for all $k \geq 3$.

Let us show now that $\lim\limits_{p\to +\infty} \mu_k^{1/p}(p) > \lim\limits_{p\to +\infty} \tau_{n(k)}^{1/p}(p)$. 
Due to Proposition \ref{prop:asym1} and Lemma  \ref{lem:tautoinfty} this inequality reads as
$2k-1 > \frac{1+\sin\left(\frac{\pi}{2n(k)}\right)}{\sin\left(\frac{\pi}{2n(k)}\right)}$, that is, $\sin\left(\frac{\pi}{2n(k)}\right) > \frac{1}{2(k-1)}$.  Using the Taylor series for the sine and the fact that $n(k) \leq \pi(k-1) - 1$, it becomes sufficient to obtain
$$
\frac{\pi}{2n(k)} - \frac{1}{3!} \left(\frac{\pi}{2n(k)}\right)^3 > \frac{\pi}{2(n(k)+1)}
$$
for all $k \geq 3$, or, equivalently, for all  $n(k) \geq 5$. However, the straightforward simplification and further analysis of the corresponding quadratic polynomial for $n(k)$ imply that this inequality holds true. This completes the proof. 

\subsection{Proof of Proposition \ref{prop:infty}}

Let $k, n \in \mathbb{N}$. Then Proposition \ref{prop:asym1} and Lemma  \ref{lem:tautoinfty} imply that equality
$\lim\limits_{p \to +\infty} \mu_k^{1/p}(p)  = \lim\limits_{p \to +\infty}\tau_n^{1/p}(p)$
can be satisfied if and only if $2k-1=\frac{1+\sin\left(\frac{\pi}{2n}\right)}{\sin\left(\frac{\pi}{2n}\right)}$. This is equivalent to the equation $\sin\left(\frac{\pi}{2n}\right) = \frac{1}{2(k-1)}$ for $k \geq 2$. However, Niven's theorem \cite[Corollary 3.12]{niven} states that
if $x/\pi$ and $\sin x$ are rational simultaneously, then $\sin x$ can take only values $0$, $\pm 1/2$ and $\pm 1$. In our case, the only possible value is $1/2$ and it is achieved by $k=2$. Consequently, $n=3$, which is the desired conclusion.

\subsection{Proof of Theorem \ref{th:not123}}
Let $\lambda_2(B_1^{\pi/4};p)$ and $\psi_p \in W_0^{1,p}(B_1^{\pi/4})$  be the second eigenvalue and a corresponding eigenfunction of the $p$-Laplacian on $B_1^{\pi/4}$ with zero Dirichlet boundary conditions.
The basic idea of the proof is to show that for sufficiently large $p$  the nodal set of $\psi_p$ contains neither radius lines of $B_1$ nor circular arcs concentric with $\partial B_1$, see the dotted lines in  Fig.~\ref{fig:Fig2}. Then the proof of \cite[Theorem 1.2]{andrabsasi} can be applied with no changes to show that $\psi_p$ generates an eigenfunction $\hat{\Psi}_{4,2}$ with corresponding eigenvalue $\hat{\tau}_{4,2}(p) = \lambda_2(B_1^{\pi/4};p)$, and, consequently, the nodal set of $\hat{\Psi}_{4,2}$ has the desired properties.

From \cite[Theorem 4.1]{julind2005} we know that
\begin{equation}\label{eq:nonsym0}
\lim_{p\to +\infty} \lambda_2^{1/p}(B_1^{\pi/4};p) 
= \frac{1}{r_2},
\end{equation}
where $r_2$ is a maximal radius of two equiradial disjoint disks inscribed in $B_1^{\pi/4}$.
It is not hard to see that the optimal configuration for $r_2$ is as in Fig.~\ref{fig:Fig2} and a straightforward calculation implies that $r_2 \approx 0.18096$.
Define the nodal set of $\psi_p$ as $\mathcal{N}_p := \{x \in B_1^{\pi/4}: \psi_p(x) = 0\}$. 
Let us denote by $B_{\rho}^{\pi/4}$ a sector of disk $B_{\rho}(0,0)$, $\rho \in (0, 1)$, such that $B_{\rho}^{\pi/4} \subset B_1^{\pi/4}$.
Suppose first that there exists a sequence $\{p_n\}_{n \in \mathbb{N}}$ such that $p_n \to +\infty$ as $n \to +\infty$ and each $\mathcal{N}_{p_n}$ contains the whole arc of $B_{\rho_n}^{\pi/4}$, $\rho_n \in (0,1)$.
Since such an arc divides $B_1^{\pi/4}$ into two subdomains $B_{\rho_n}^{\pi/4}$ and $B_1^{\pi/4} \setminus \overline{B_{\rho_n}^{\pi/4}}$, and the second eigenfunction $\psi_{p_n}$ has exactly two nodal domains \cite{cuesta}, we conclude that, without loss of generality, $\psi_{p_n} = \psi_{p_n}^+$ in $B_{\rho_n}^{\pi/4}$ and $\psi_{p_n} = -\psi_{p_n}^-$ in $B_1^{\pi/4} \setminus \overline{B_{\rho_n}^{\pi/4}}$, where $\psi_{p_n}^\pm = \max\{\pm \psi_{p_n},0\}$. 
Moreover, by assumption, $\psi_{p_n}^+$ and $\psi_{p_n}^-$  are nonnegative eigenfunctions not only in their supports, but also in the whole subdomains $B_{\rho_n}^{\pi/4}$ and $B_1^{\pi/4} \setminus \overline{B_{\rho_n}^{\pi/4}}$, respectively, and they correspond to the eigenvalue $\lambda = \lambda_2(B_1^{\pi/4};{p_n})$.
Therefore, the domain monotonicity of the first eigenvalue of the $p$-Laplacian yields 
$$
\lambda_2(B_1^{\pi/4};{p_n}) \geq  \lambda_1(B_{\rho_n}^{\pi/4};{p_n})
\quad \text{and} \quad 
\lambda_2(B_1^{\pi/4};{p_n}) \geq \lambda_1(B_1^{\pi/4} \setminus \overline{B_{\rho_n}^{\pi/4}};{p_n}).
$$
On the other hand, any eigenfunction except the first one has to be sign-changing \cite{ananeeig}. Thus, the only possibility is
\begin{equation}\label{eq:nonsym1}
\lambda_2(B_1^{\pi/4};{p_n}) =  \lambda_1(B_{\rho_n}^{\pi/4};{p_n}) = \lambda_1(B_1^{\pi/4} \setminus \overline{B_{\rho_n}^{\pi/4}};{p_n}).
\end{equation}
Hence, \eqref{eq:nonsym0} and \eqref{eq:nonsym1} imply that
\begin{equation}\label{eq:1/r}
\lim_{n \to +\infty} \lambda_1^{1/{p_n}}(B_{\rho_n}^{\pi/4};{p_n}) = \frac{1}{r_2}
\quad
\text{ and }
\quad 
\lim_{n \to +\infty} \lambda_1^{1/{p_n}}(B_1^{\pi/4} \setminus \overline{B_{\rho_n}^{\pi/4}};{p_n}) = 
\frac{1}{r_2}.
\end{equation}

Since $\rho_n \in (0,1)$ for all $n \in \mathbb{N}$, we can find a monotone subsequence, denoted again by 
$\{\rho_n\}_{n \in \mathbb{N}}$, and a number $\rho \in [0,1]$, such that $\rho_n \to \rho$ as $n \to +\infty$. From \eqref{eq:nonsym1} we immediately deduce that $\rho \in (0,1)$.
Assume that $\rho_n$ is increasing, that is, $B_{\rho_m}^{\pi/4} \subset B_{\rho_n}^{\pi/4} \subset B_{\rho}^{\pi/4}$ for all $n,m \in \mathbb{N}$ such that $m < n$.  Then, for all such $n,m$ the domain monotonicity  implies that
\begin{align*}
\lambda_1^{1/{p_n}}(B^{\pi/4}_{\rho_m};p_n)
\geq
\lambda_1^{1/{p_n}}(B^{\pi/4}_{\rho_n};p_n)
\geq
\lambda_1^{1/{p_n}}(B^{\pi/4}_{\rho};p_n).
\end{align*}
Passing to the limit as $n \to +\infty$ and using formulas \eqref{eq:p-to-infty} and \eqref{eq:1/r} we get
\begin{equation}\label{eq:rrr}
\frac{1}{r_{B_{\rho_m}^{\pi/4}}} \geq 
\frac{1}{r_2} \geq 
\frac{1}{r_{B_{\rho}^{\pi/4}}}
\end{equation}
for all $m \in \mathbb{N}$, where $r_{B_{\rho_m}^{\pi/4}}$ and $r_{B_{\rho}^{\pi/4}}$ are the radii of  maximal disks inscribed in $B_{\rho_m}^{\pi/4}$ and $B_{\rho}^{\pi/4}$, respectively. 
Taking now the limit in \eqref{eq:rrr} as $m \to +\infty$ and noting that $r_{B_{\rho_m}^{\pi/4}} \to r_{B_{\rho}^{\pi/4}}$, we conclude that 
$r_2 = r_{B_{\rho}^{\pi/4}}$.
If $\rho_n$ is decreasing,  we proceed analogously and derive the same equality. 
Likewise, we can obtain that $r_2 = r_{B_1^{\pi/4} \setminus \overline{B_{\rho}^{\pi/4}}}$.

On the other hand, explicit computations show that
the maximal radius $r_{circ}$ of two equiradial disjoint disks $B_r^1$ and
$B_r^2$, such that $B_r^1$ is inscribed in $B_\rho^{\pi/4}$ and $B_r^2$ is inscribed in $B_1^{\pi/4} \setminus \overline{B_\rho^{\pi/4}}$, among all $\rho \in (0,1)$, is equal to $r_{circ}=\frac{\sin(\pi/8)}{1+3\sin(\pi/8)} \approx 0.17815 < r_2 \approx 0.18096$. A contradiction.

The same arguments can be applied to prove that the nodal set $\mathcal{N}_{p}$ does not contain a radial line of $B_1$ for all $p$ sufficiently large. Indeed, the domain monotonicity implies that the best choice for this line is to be a bisector of $B_1^{\pi/4}$. Nevertheless, the maximal radius of a disk inscribed in $B_1^{\pi/8}$ is given by 
$r_{rad}=\frac{\sin(\pi/16)}{1+\sin(\pi/16)} \approx 0.16324$ which is again strictly less than $r_2$.
This is a contradiction and the proof is finished.

\begin{remark}
	The behavior of the nodal set of the second eigenfunction in a sector with angle $\pi/4$ described in the proof of Theorem \ref{th:not123} is reminiscent of the numerical results for an isosceles triangle with base $1$ and height $1$, which were obtained in \cite{horak}.
\end{remark}

\medskip
\noindent
\textbf{Acknowledgments.}
The first author was supported by the project LO1506 of the Czech Ministry of Education, Youth and Sports, the second author by the grant 13-00863S of the Grant Agency of the Czech Republic.

\end{document}